\documentclass[11pt]{article}

\usepackage{tikz}
\usepackage{subfigure}
\usepackage[english]{babel}
\usepackage[ruled,algosection,linesnumbered]{algorithm2e}
\usepackage[center]{caption2}
\usepackage{amsfonts,amssymb,amsmath,latexsym,amsthm}
\usepackage{multirow}
\usepackage[usenames,dvipsnames]{pstricks}
\usepackage{epsfig}
\usepackage{pst-grad} 
\usepackage{pst-plot} 
\usepackage[space]{grffile} 
\usepackage{etoolbox} 
\usepackage{indentfirst}
\makeatletter 
\patchcmd\Gread@eps{\@inputcheck#1 }{\@inputcheck"#1"\relax}{}{}

\newtheorem{thm}{Theorem}

\newtheorem{lem}[thm]{Lemma}

\newtheorem{conj}[thm]{Conjecture}

\newtheorem{rem}{Remark}

\begin{document}

\title{Strongly connected orientation with minimum lexicographic order of indegrees
\thanks{The work was supported by NNSF of China (No. 11671376) and  NSF of Anhui Province (No. 1708085MA18) and Anhui Initiative in Quantum Information Technologies (AHY150200).}
}
\author{Hongyu Zhou$^a$, \quad Xinmin Hou$^b$\\
\small $^{a,b}$ Key Laboratory of Wu Wen-Tsun Mathematics\\
\small School of Mathematical Sciences\\
\small University of Science and Technology of China\\
\small Hefei, Anhui 230026, China.
}

\date{}

\maketitle

\begin{abstract}
Given a simple undirected graph $G$, an orientation of $G$ is to assign every edge of $G$ a direction. Borradaile et al gave a greedy algorithm SC-Path-Reversal (in polynomial time)  which finds a strongly connected orientation that minimizes the maximum indegree, and conjectured that SC-Path-Reversal is indeed optimal for the "minimizing the lexicographic order" objective as well. In this note, we give a positive answer to the conjecture, that is we show that the algorithm SC-PATH-REVERSAL finds a strongly connected orientation that minimizes the lexicographic order of indegrees.
\end{abstract}

\section{Introduction}
Graph orientation has long been studied and is a rich field under different conditions. In this note we mainly concern about the strongly-connected orientation with minimum lexicographic order. This objective arises from a telecommunication network design problem~\cite{Gb17,Ven04}.
Let $G=(V, E)$ be an undirected simple graph.
An {\it orientation} $\Lambda$ of $G$ is an assignment to each edge a direction. By a {\it strongly-connected orientation}, we mean the digraph that we obtain is strongly-connected. In a digraph $D=(V,\Lambda)$,
the {\it indegree} of a vertex $v$ is the number of arcs that are directed to $v$, denoted by $d^-_\Lambda(v)$.
The {\it indegree sequence} of a digraph (or an orientation) is defined as a non-increasing sequence of the indegrees of all the vertices, that is, we place the indegrees of vertices in a non-increasing order. To compare distinct indegree sequences of two orientations of an undirected graph, we apply the lexicographic order, i.e. let
$s=(s_1, s_2,\ldots,s_n)$ and $t=(t_1, t_2,\ldots, t_n)$ be the indegree sequences of two distinct orientations of an undirected graph, respectively, we say $s$ is smaller than $t$ if there exists an integer $k$ with  $1\le k\le n$ such that $s_k<t_k$ and $s_i=t_i$ for all $i<k$, and vise versa.



Let  $D=(V,\Lambda)$ be a digraph and $u,v\in V$, we say $u$ {\it two-reaches} to $v$ (or $v$ is {\it two-reachable} from $u$) if there are two arc-disjoint directed paths from $u$ to $v$ in $D$.
A directed path from $u$ to $v$ is called {\it reversible} if $d^-(u)<d^-(v)-1$, and is called {\it strongly reversible} if $d^-(u)<d^-(v)-1$ and $u$ two-reaches $v$ in $D$.

The following greedy algorithm was first given by de Fraysseix and de Mendez~\cite{FM95}.
\begin{algorithm}\label{ALG: PR}
\small
\SetAlgoLined
\KwIn{Undirected simple graph $G$}
\KwOut{Orientation $\Lambda$}

Arbitrarily orient every edge of $G$.

While there is a reversible path, reverse it.

Repeat step 2.
\caption{PATH-REVERSAL}
\end{algorithm}
It has been shown that the algorithm finds an orientation that minimizes the maximum indegree, which is proved by Venkateswaran~\cite{Ven04},  Asahiro et al.~\cite{AMOZ07}, and de Fraysseix and de Mendez~\cite{FM95}, respectively. In fact, Path-Reversal can do something more, Borradaile et al.~\cite{Gb17} showed that Path-Reversal indeed finds an orientation that minimizes the lexicographic order of the indegrees.

\begin{lem}[Lemma 2 in~\cite{Gb17}]\label{LEM: l1}
Reversing a directed path from $u$ to $v$ maintains the strong connectivity if and only if $u$ two-reaches $v$. Particularly, $v$ two-reaches $v$ itself.
\end{lem}

By Lemma~\ref{LEM: l1}, we know that reversing a strongly reversible direct path remains the resulting digraph strongly-connected. Based on PATH-REVERSAL, Borradaile et al.~\cite{Gb17} gave a modified version of it as shown in the following.


\begin{algorithm}\label{ALG: SC_PR}
\small
\SetAlgoLined
\KwIn{Undirected simple graph $G$ admitting a strongly connected orientation}
\KwOut{Orientation $\Lambda$}

Find an arbitrary strongly-connected orientation.

If there exists a strongly reversible path, reverse it.

Repeat step 2.
\caption{SC-PATH-REVERSAL}
\end{algorithm}

Borradaile et al ~\cite{Gb17} showed that the algorithm SC-Path-Reversal finds (in polynomial time) a strongly connected
orientation that minimizes the maximum indegree, and conjectured that SC-Path-Reversal is indeed optimal for the the "minimizing the lexicographic order" objective as well.
\begin{conj}[Borradaile et al ~\cite{Gb17}]\label{CONJ: c1}
The algorithm SC-Path-Reversal finds a strongly connected orientation that minimizes the lexicographic order of indegrees.
 \end{conj}

In this note, we give a positive answer to Conjecture~\ref{CONJ: c1}. The following is our main theorem.

\begin{thm}\label{THM: main}
The algorithm SC-PATH-REVERSAL  finds an orientation that minimizes the lexicographic order of the indegrees.
\end{thm}

In the rest of the note, we give the proof of Theorem~\ref{THM: main}.


\section{Proof of Theorem~\ref{THM: main}}
Before proving that the orientation output has the minimum lexicographic order, we introduce two lemmas given in~\cite{Gb17}.

\begin{lem}[Lemma 2 in~\cite{Gb17}]\label{LEM: l2}
In a digraph, let $s$ and $t$ be two vertices (can be identical) that 2-reach a vertex $v$. If there exists a vertex $u$ such that one $u\rightarrow s$ and one $u\rightarrow t$ paths are arc-disjoint, then $u$ two-reaches $v$.
\end{lem}


The following lemma is a variable version of Lemma~\ref{LEM: l2} in~\cite{Gb17}, the proof is the same as the one of that lemma, so we omit it here.
\begin{lem}\label{LEM: l3}
Let $v$ be a vertex in a strongly connected digraph $D=(V,\Lambda)$ and let $U$ be the set of vertices that 2-reach $v$. Then for any component $C$ of $D[V\setminus U]$, there is exactly one arc from $C$ to $U$.
\end{lem}

\begin{rem}
By a component we mean a connected component instead of a strongly-connected component.
\end{rem}




Now we are ready to give the proof of the main theorem.


\begin{proof}[Proof of Theorem~\ref{THM: main}]
Let $G=(V,E)$ be the underlying graph. Let $D_{pr}=(V,\Lambda_{pr})$ be the orientation founded by the SC-PATH-REVERSAL, and let $D_{lex}=(V,\Lambda_{lex})$ be a strongly connected orientation that minimizes the lexicographic order of indegrees among all strongly connected orientations of $G$. Write $d^-_{pr}(v)$ and $d^-_{lex}(v)$ for the indegree of $v$ in $D_{pr}$  and  $D_{lex}$, respectively.
Define
$$\Delta:=\sum_{v\in V}|d^-_{lex}(v)-d^-_{pr}(v)|$$
and $$S:=\{\, v\, |\, d^-_{lex}(v)\neq d^-_{pr}(v) \}.$$
Choose $\Lambda_{lex}$ such that it minimizes $\Delta$.
If $S=\emptyset$, then $\Lambda_{pr}$ is a strongly-connected orientation having the same lexicographic order of indegree as $\Lambda_{lex}$, we are done. So assume $S\neq\emptyset$.
Denote $M_1=\max \{\, d^-_{lex}(v)\, |\, v\in S\}$ and $S_1=\{\, u\in S\, |\, d_{lex}^-(u)=M_1 \}$. Denote $M_2=\max \{\, d^-_{pr}(u)\, |\, u\in S_1\}$.
Choose $v\in S_1$ such that $d_{pr}^-(v)=M_2$.

\vspace{5pt}
\noindent{\bf Case 1:} $d^-_{lex}(v)>d^-_{pr}(v)$.

Let $U$ be the set of vertices that two-reach $v$ in $D_{lex}$. By Lemma~\ref{LEM: l3}, there is exactly one arc from each component of $G[V\setminus U]$ to $U$. Thus, on the one hand,
$$\sum_{u\in U}d^-_{lex}(u)=|E(G[U])|+c(G[V\setminus U]),$$
where $c(G)$ denotes the number of components of a graph $G$.

While, on the other hand, since $D_{pr}$ is strongly connected, there is at least one arc from each component of $G[V\setminus U]$ to $U$. Thus
$$\sum_{u\in U}d^-_{pr}(u)\ge |E(G[U])|+c(G[V\setminus U]).$$
So we get $\sum\limits_{u\in U}d^-_{lex}(u)\le \sum\limits_{u\in U}d^-_{pr}(u)$. Since $d^-_{lex}(v)>d^-_{pr}(v)$, there exists a vertex $w\in U$ such that $d^-_{lex}(w)<d^-_{pr}(w)$.  Clearly, $w\in S$. By the choice of $v$, we have $d^-_{lex}(w)\le d^-_{lex}(v)$.
We claim that $d^-_{lex}(w)< d^-_{lex}(v)$.
  If not, $w\in S_1$. But $d^-_{pr}(w)>d^-_{lex}(w)=d^-_{lex}(v)>d^-_{pr}(v)$, a contradiction to the choice of $v$.
If $d^-_{lex}(w)<d^-_{lex}(v)-1$, then reversing a directed  path from $w$ to $v$ remains strong connectivity by Lemma~\ref{LEM: l1}, but the resulting orientation has a smaller lexicographic order of indegree,
a contradiction to the choice of $\Lambda_{lex}$. Thus $d^-_{lex}(w)=d^-_{lex}(v)-1$. Now, reversing a directed path from $w$ to $v$ in $D_{lex}$, we get another orientation with minimum lexicographic order, and however, with $d^-_{lex}(v)>d^-_{pr}(v)$ and  $d^-_{lex}(w)<d^-_{pr}(w)$ before the reverse, we get a smaller $\Delta$, which contradicts to the choice of $\Lambda_{lex}$, too.

\vspace{5pt}
\noindent{\bf Case 2:} $d^-_{lex}(v)<d^-_{pr}(v)$.

Let $U$ be the set of vertices that 2-reach $v$ in $D_{pr}$. With a similar discussion as in Case 1, we get $\sum\limits_{u\in U}d^-_{lex}(u)\ge \sum\limits_{u\in U}d^-_{pr}(u)$. Since $d^-_{lex}(v)<d^-_{pr}(v)$, there exists a vertex $w\in U$ such that $d^-_{lex}(w)>d^-_{pr}(w)$.
Then we have
$$d^-_{pr}(v)>d^-_{lev}(v)\ge d^-_{lex}(w)>d^-_{pr}(w).$$
This implies that $d^-_{pr}(w)<d^-_{pr}(v)-1$. So there is  a  strongly reversible directed path in $D_{pr}$, which contradicts to the property $D_{pr}$ has no strongly reversible directed path.

We conclude that $S=\emptyset$. Therefore the orientation $\Lambda_{pr}$ found by the SC-PATH-REVERSAL algorithm has the minimum lexicographic order of indegree.
\end{proof}


\appendix
\section{Proof of Lemma~\ref{LEM: l3} }

Let $C$ be a component of $G[V\setminus U]$. Note that there is no arc between $C$ and the other components of $G[V\setminus U]$.
Since $D$ is strongly connected, there is at least one arc from $C$ to $U$. 
Let $v_1,\ldots,v_p$ be the tails of the arcs from  $C$ to $U$ and let $W_k$ be the set of vertices in $C$ that reach $v_k$ for $k=1,\ldots,p$. Note that there is exactly one arc from $v_i$ to $U$ for $i=1,\ldots, p$, otherwise $v_i$ 2-reaches $v$ by Lemma~\ref{LEM: l2}, which is a contradiction to the choice of $U$. So in the following it is sufficient to show that $p=1$. Denote the head vertex of the arc from $v_i$ to $U$ by $u_i$, $i=1, \ldots, p$. 
We first claim that $W_i\cap W_j\neq\emptyset$ for any pair of different $W_i$ and $W_j$. If not, suppose there is $x\in W_i\cap W_j$, then there exists a directed path $P_i$ from $x$ to $v_i$ and a directed path $P_j$ from $x$ to $v_j$ in $C$. 
Let $y$ be the last common vertex in $V(P_i)\cap V(P_j)$ along the direction of $P_i$. Then $y\in W_i\cap W_j$ and there are two arc-disjoint directed paths  from $y$ to $v_i$ and $v_j$ in $C$ and hence from $y$ to $u_i$ and $u_j$ in $G$, respectively. 
By Lemma~\ref{LEM: l2}, $y$ 2-reaches $v$, a contradiction to $y\not\in U$. Thus $W_i\cap W_j=\emptyset$ for any $i\neq j$. The claim also implies that there is no arc between $W_i$ and $W_j$. However, $C$ is connected (not necessarily strongly connected). This forces that $p=1$. The proof is complete. 

\end{document}